\documentclass[submission]{FPSAC2020}

\usepackage{amssymb,amsmath}
\usepackage{graphicx,xspace}
\usepackage{epsfig}
\usepackage{tikz}
\usepackage[T1]{fontenc}
\usepackage[utf8]{inputenc}
\usepackage{comment}
\usepackage{tabularx}

\usepackage{ bbold }

\usepackage[shortlabels]{enumitem}
\usepackage{bm}
\usepackage{cancel}

\usepackage{hyperref,pifont}
\usepackage{todonotes}
\usepackage{mathtools}
\usepackage{dsfont}
\usepackage{blkarray, bigstrut}

\usepackage[capitalize]{cleveref}

\theoremstyle{plain}

\newtheorem{lemma}{Lemma}[section]
\newtheorem{theorem}[lemma]{Theorem}

\newtheorem{proposition}[lemma]{Proposition}
\newtheorem{definition}[lemma]{Definition}

\theoremstyle{remark}

\newtheorem{conjecture}[lemma]{Conjecture}

\newtheorem{observation}[lemma]{Observation}

\def\CCC{\mathcal{C}}  %
\def\SS{\mathcal{S}}   %
\def\NN{\mathbb N}

\def\R{\mathbb{R}}

\def\sc{\mathcal{C}}

\DeclareMathOperator{\occ}{occ}

\DeclareMathOperator{\pat}{pat}

\DeclareMathOperator{\conv}{conv}

\DeclareMathOperator{\coc}{c-occ} %
\DeclareMathOperator{\oc}{occ} %

\def\Z{\mathbb{Z}}

\def\pcoc{\widetilde{\coc}}
\def\poc{\widetilde{\oc}}
\def\JacGraph[#1]{G_J(#1)}
\def\ValGraph[#1]{\mathcal{O}v(#1)}
\def\vecc[#1]{\vec{e}_{\CCC_{#1}}}

\newcounter{indice}

\title[The feasible region for consecutive patterns of permutations is a cycle polytope]{The feasible region for consecutive patterns of permutations is a cycle polytope}

\author[ J. Borga and R. Penaguiao]{Jacopo Borga\thanks{\href{mailto:jacopo.borga@math.uzh.ch}{jacopo.borga@math.uzh.ch}. Jacopo Borga was supported by the SNF grant 200021-172536}\addressmark{1}, \and Raul Penaguiao\thanks{\href{mailto:raul.penaguiao@math.uzh.ch}{raul.penaguiao@math.uzh.ch}. Raul Penaguiao was supported by the SNF grant 200020-172515} \addressmark{1}}

\address{\addressmark{1}Department of Mathematics, University of Zurich, Switzerland}

\received{20th November 2019}

\abstract{
We study proportions of consecutive occurrences of permutations of a given size. Specifically, the feasible limits of such proportions on large permutations form a region, called \emph{feasible region}. We show that this feasible region is a polytope, more precisely the cycle polytope of a specific graph called \emph{overlap graph}. This allows us to compute the dimension, vertices and faces of the polytope.
Finally, we prove that the limits of classical occurrences and consecutive occurrences are independent, in some sense made precise in the extended abstract. As a consequence, the scaling limit of a sequence of permutations induces no constraints on the local limit and \textit{vice versa}.
}

\keywords{permutation patterns, cycle polytopes, overlap graphs.}

\begin{document}

\maketitle

This is a shorter version of the preprint \cite{borga2019feasible} that is currently submitted to a journal.
Many proofs and details omitted here can be found in \cite{borga2019feasible}.

\begin{figure}[h]\centering
	\includegraphics[scale=0.57]{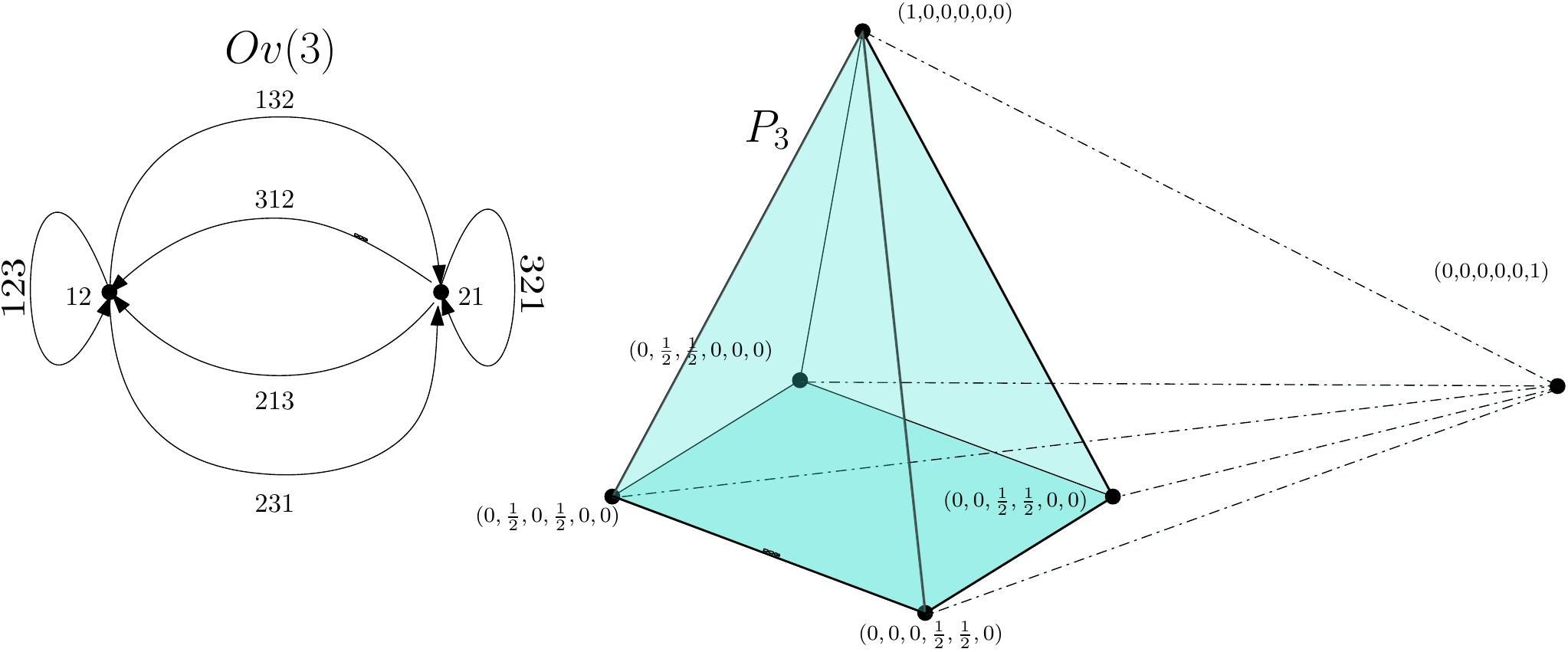}
	\bigskip
	\caption{\textbf{Left:} The overlap graph $\ValGraph[3]$ (see \cref{defn:ov_graph} for a precise definition).\break \textbf{Right:} The four-dimensional polytope $P_3$ given by the six patterns of size three (see \eqref{eq:setofinter} for a precise definition). We highlight in light-blue one of the six three-dimensional faces of $P_3$. This face is a pyramid with square base. The polytope itself is a four-dimensional pyramid, whose base is the highlighted face. From \cref{thm:main_res} we have that $P_3$ is the cycle polytope of $\ValGraph[3]$. \label{fig:P_3}}
\end{figure}

\section{Introduction}
\label{sect:intro}

\subsection{Motivations}

\label{sect:occ and cocc}

Despite not presenting any probabilistic result here, we give some motivations that come from the study of random permutations.
This is a classical topic at the interface of combinatorics and discrete probability theory. There are two main approaches to it: the first concerns the study of statistics on permutations, and the second, more recent, looks for the limits of permutations themselves. The two approaches are not orthogonal and many results relate them, for instance \cref{thm:carac_occ,thm:carac_cocc} below.  

In order to study the limit of permutations, two main notions of convergence have been defined: a global notion of convergence (called permuton convergence) and a local notion of convergence (called Benjamini--Schramm convergence, or BS-convergence, for short).
A permuton is a probability measure on the unit square with uniform marginals.
The notion of permuton limit for permutations has been introduced in \cite{MR2995721}, and represents the scaling limit of a permutation seen as a permutation matrix, as the size grows to infinity. The study of permuton limits is an active and exciting research field in combinatorics, see for instance \cite{bassino2019scaling, MR3813988,kenyon2015permutations} and references therein.
On the other hand, the notion of BS-limit for permutations is more recent, and it has been introduced in \cite{borga2018local}. 
Informally, in order to investigate BS-limits, we study the permutation in a neighborhood around a randomly marked point in its one-line notation. Limiting objects for this framework are called \emph{infinite rooted permutations} and are in bijection with total orders on the set of integer numbers. 
BS-limits have also been studied in some other works, see for instance \cite{bevan2019permutations,borga2019square,borga2019decorated}.

Let $n\in\NN=\Z_{>0}$. We denote by $\mathcal{S}_n$ the set of permutations of size $n$, and by $\SS$ the set of all permutations.
We recall the definition of patterns for permutations. For $\sigma\in \SS_n,$ $\pi \in \SS_k$ and a subsequence $1\leq i_1 < \ldots < i_k\leq n$, we say that $\sigma(i_1) \ldots \sigma(i_k)$ is an \emph{occurrence} of $\pi$ in $\sigma$, if $\sigma(i_1) \ldots \sigma(i_k)$ has the same relative order as $\pi.$ If the indices $i_1, \ldots ,i_k$ form an interval, then we say that $\sigma(i_1) \ldots \sigma(i_k)$ is a \emph{consecutive occurrence} of $\pi$ in $\sigma$.
We denote by $\oc(\pi,\sigma)$ (resp. by $\coc(\pi,\sigma)$) the number of occurrences (resp. consecutive occurrences) of a pattern $\pi$ in $\sigma$.
For example, if $\pi = 132$ and $\sigma = 72638145 $, we have that $\oc(\pi,\sigma) = 7 $ and $\coc(\pi,\sigma) = 1$.
We also denote by $\widetilde{\occ}(\pi,\sigma)$ (resp.\ $\pcoc(\pi,\sigma)$) the proportion of classical occurrences (resp.\ consecutive occurrences) of a permutation $\pi\in\SS_k$ in $\sigma\in\SS_n$, that is 
\begin{equation*}
\label{conpatden}
\poc(\pi,\sigma)\coloneqq\frac{\oc(\pi,\sigma)}{\binom{n}{k}}\in[0,1] , \quad \quad \pcoc(\pi,\sigma)\coloneqq\frac{\coc(\pi,\sigma)}{n}\in[0,1]\, .
\end{equation*}
The following theorems provide relevant combinatorial characterizations of the two aforementioned notions of convergence for permutations.

\begin{theorem}[\cite{MR2995721}]\label{thm:carac_occ}
	For any $n\in\NN$, let $\sigma^n\in\SS$ and assume that $|\sigma^n|\to\infty$. The sequence $(\sigma^n)_{n\in\NN}$ converges to some permuton $P$ if and only if there exists a vector $(\Delta_{\pi}(P))_{\pi\in\mathcal{S}}$ of non-negative real numbers (that depends on $P$) such that, for all $\pi\in\mathcal{S}$,
	$$\widetilde{\occ}(\pi,\sigma^n)\to\Delta_{\pi}(P).$$
\end{theorem}

\begin{theorem}[\cite{borga2018local}]\label{thm:carac_cocc}
	For any $n\in\NN$, let $\sigma^n\in\SS$ and assume that $|\sigma^n|\to\infty$. The sequence $(\sigma^n)_{n\in\NN}$ converges in the Benjamini--Schramm topology to some random infinite rooted permutation $\sigma^\infty$ if and only if there exists a vector $(\Gamma_{\pi}(\sigma^\infty))_{\pi\in\mathcal{S}}$ of non-negative real numbers (that depends on $\sigma^\infty$) such that, for all $\pi\in\mathcal{S}$,
	$$\widetilde{\coc}(\pi,\sigma^n)\to\Gamma_{\pi}(\sigma^\infty).$$
\end{theorem}

A natural question, motivated by the theorems above, is the following: given a finite family of patterns $\mathcal{A}\subseteq\SS$ and a vector $(\Delta_\pi)_{\pi\in\mathcal{A}}\in [0,1]^{\mathcal A}$, or $(\Gamma_\pi)_{\pi\in\mathcal{A}}\in [0,1]^{\mathcal A}$, does there exist a sequence of permutations $(\sigma^n)_{n\in\NN}$ such that $|\sigma^n|\to\infty$ and
$$\widetilde{\occ}(\pi,\sigma^n)\to\Delta_{\pi}, \quad \text{for all} \quad \pi\in\mathcal{A},$$
or
$$\widetilde{\coc}(\pi,\sigma^n)\to\Gamma_{\pi}, \quad \text{for all} \quad \pi\in\mathcal{A}\;?$$

We consider the classical pattern limiting sets, sometimes called the \textit{feasible region} for (classical) patterns, defined as
\begin{align}
\label{eq:setofnotinter}
clP_k \coloneqq&\left\{\vec{v}\in [0,1]^{\SS_k} \big| \exists (\sigma^m)_{m\in\NN} \in \SS^{\NN} \text{ s.t. }|\sigma^m| \to \infty\text{ and } \poc(\pi, \sigma^m ) \to \vec{v}_{\pi},\forall \pi\in\SS_k  \right\}\\
=&\left\{(\Delta_{\pi}(P))_{\pi\in\SS_k} \big| P\text{ is a permuton}  \right\} \, ,\nonumber
\end{align}
and we introduce the consecutive pattern limiting sets, called here the \textit{feasible region} for consecutive patterns,
\begin{align}
\label{eq:setofinter}
P_k \coloneqq &\left\{\vec{v}\in [0,1]^{\SS_k} \big| \exists (\sigma^m)_{m\in\NN} \in \SS^{\NN} \text{ s.t. }|\sigma^m| \to
\infty \text{ and }  \pcoc(\pi, \sigma^m ) \to \vec{v}_{\pi}, \forall \pi\in\SS_k \right\}\\
=&\left\{(\Gamma_{\pi}(\sigma^{\infty}))_{\pi\in\SS_k} \big| \sigma^{\infty}\text{ is a random infinite rooted \emph{shift-invariant} permutation}  \right\} \, . \nonumber
\end{align}

For the precise definition of \emph{shift-invariant} permutations see \cite[Defintion 2.41]{borga2018local}. The equalities in \eqref{eq:setofnotinter} and \eqref{eq:setofinter} follow from \cite[Theorem 1.6]{MR2995721} and \cite[Proposition 3.4]{borga2019feasible} respectively.

The feasible region $clP_k$ was previously studied in several papers (see \cref{sect:feas_reg}). 
The main goal of this project is to analyze the feasible region $P_k$, that turns out to be related to specific graphs called \emph{overlap graphs} (see \cref{sect:ov_gr}) and its corresponding cycle polytope (see \cref{sect:cyc_pol}).
In particular, the feasible region $P_k$ is convex.
This is not in general true for the region $clP_k$.

\subsection{Definitions and informal statement of the main results}

To introduce our main results we need two key definitions.

\begin{definition}
	\label{defn:ov_graph}
	The \emph{overlap graph} $\ValGraph[k]$ is a directed multigraph with labeled edges, where the vertices are elements of $\SS_{k-1}$ and for every $\pi\in\SS_{k}$ there is an edge labeled by $\pi$ from the pattern induced by the first $k-1$ indices of $\pi$ to the pattern induced by the last $k-1$ indices of $\pi$.
\end{definition}

The overlap graphs  $\ValGraph[3]$ and $\ValGraph[4]$ are displayed in \cref{fig:P_3} and \cref{Overlap_graph_exemp} respectively.
\begin{figure}
	\begin{minipage}[c]{0.49\textwidth}
		\centering
		\includegraphics[width=0.8\textwidth]{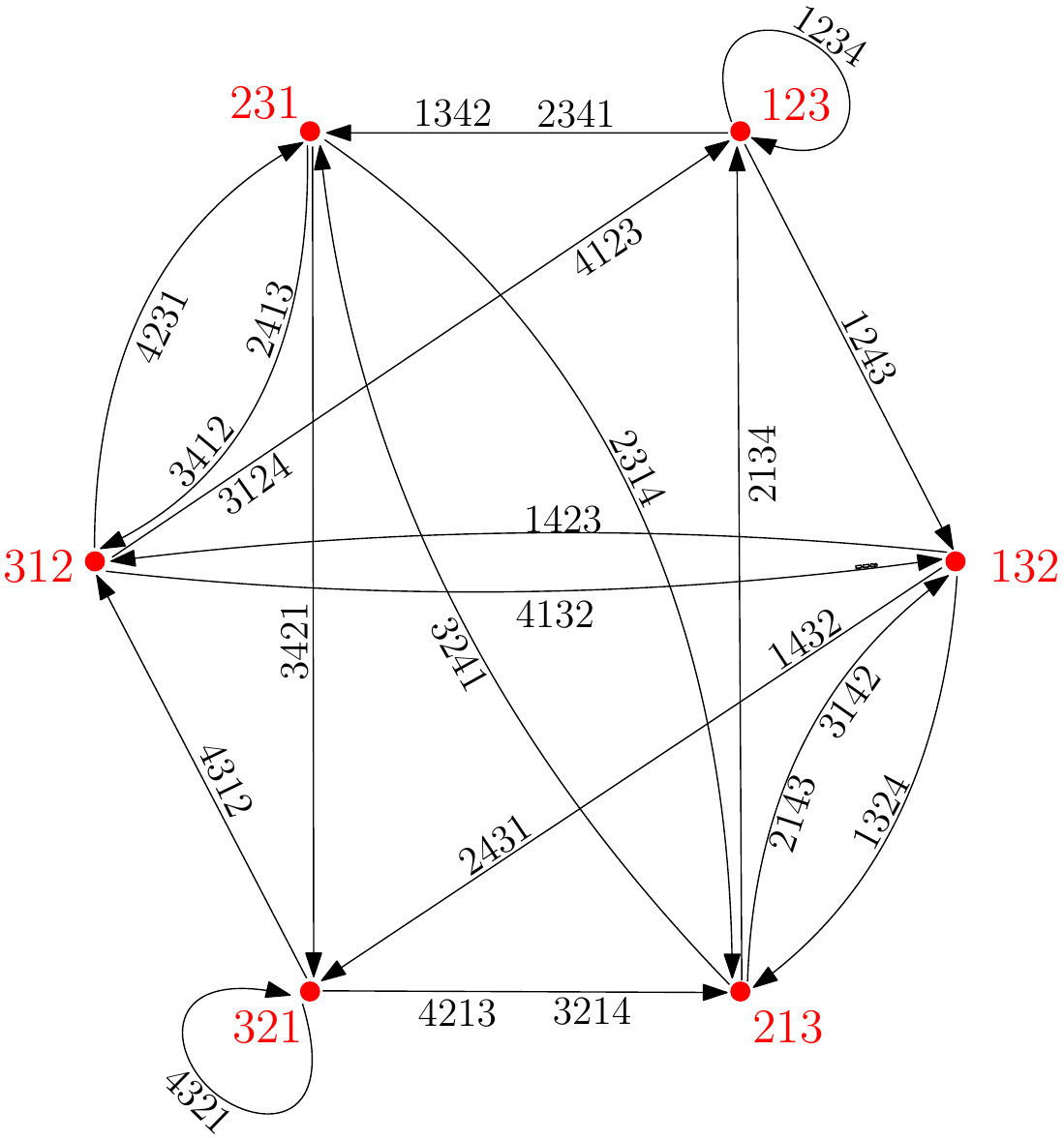}
	\end{minipage}
	\begin{minipage}[c]{0.49\textwidth}
		\caption{
			The overlap graph $\ValGraph[4]$. The six vertices are painted in red and the edges are drawn as labeled arrows. Note that in order to obtain a clearer picture we did not draw multiple edges, but we use multiple labels (for example the edge $231 \to 312$ is labeled with the permutations $3412$ and $2413$ and should be thought of as two distinct edges labeled with $3412$ and $2413$ respectively). \label{Overlap_graph_exemp}
		}
	\end{minipage}
\end{figure}

We denote the convex hull of a family $\mathcal F$ of points by $\conv \mathcal F$.
A non-empty cycle $\mathcal C$ in a directed multigraph is said to be a \textit{simple cycle} if it does not repeat vertices.

\begin{definition}
	\label{def:cyclepoly}
	Let $G=(V,E)$ be a directed multigraph.
	For each non-empty cycle $\mathcal{C}$ in $G$, define $\vec{e}_{\mathcal{C}}\in \mathbb{R}^{E}$ so that
	\begin{equation}\label{eq:cycvec}
	(\vec{e}_{\mathcal{C}})_e \coloneqq \frac{\text{number of occurrences of $e$ in $\sc$}}{|\mathcal{C}|}, \quad  \text{for all} \quad e\in E. 
	\end{equation}
	We define the \emph{cycle polytope} of $G$ to be $P(G) \coloneqq \conv \{\vec{e}_{\mathcal{C}} | \, \mathcal{C} \text{ is a simple cycle of } G \}$.
\end{definition}

The example of the cycle polytope $P(\ValGraph[3])$ is given in \cref{fig:P_3}.

\medskip

Our first result is a full description of the feasible region $P_k$ as the cycle polytope of the overlap graph $\ValGraph[k]$ (see \cref{thm:main_res}).

The second result deals with the dimension and the equations that define the cycle polytope of a general directed multigraph (see \cref{thm:dimfullgr}) and a description of the associated face poset (see \cref{cor:facestruct}).
From this theorem, the dimension and the equations of $P_k$ can be directly computed.

Our last result addresses a slightly different question. We intertwine the notions of classical patterns and consecutive patterns, and discuss the shape of the resulting feasible region. 
In \cref{thm:mixing}, we reduce the description of this new feasible region to the independent descriptions of the feasible regions for pattens and consecutive patterns discussed in \cref{sect:occ and cocc}.

\subsection{State-of-the-art and a conjecture}

\subsubsection{The feasible region for classical patterns}
\label{sect:feas_reg}
The feasible region $clP_k$ was first studied in \cite{kenyon2015permutations} for some particular families of patterns instead of the whole $\SS_k$. The authors specifically studied patterns of small size (2 or 3) showing that, already in these cases, the problem of describing the associated feasible region presents a lot of difficulties. They also showed that this region is not always convex (see for instance \cite[Section 9]{kenyon2015permutations}).

\medskip

The set $clP_k$ was also studied in \cite{MR3567538}, even though with a different objective.
There, it was shown that $clP_k$ contains an open ball $B$ with dimension $|I_k|$, where $I_k$ is the set of permutations of size at most $k$ that are $\oplus$-indecomposable.
Specifically, for a particular ball $B\subseteq \R^{I_k}$, the authors constructed permutons $P_{\vec{x}}$ such that $\Delta_{\pi }(P_{\vec{x}}) = \vec{x}_{\pi}$, for each point $\vec{x} \in B$ and each $\pi\in I_k$.

The work done in \cite{MR3567538} opened the problem of finding the maximal dimension of an open ball contained in $clP_k$, and placed a lower bound on it.
In \cite{vargas2014hopf} an upper bound for this maximal dimension is indirectly given as the number of so-called \textit{Lyndon permutations} of size at most $k$, whose set we denote $\mathcal{L}_k$.
The author shows that for any permutation $\pi$ that is not a Lyndon permutation, $\poc(\pi, \sigma ) $ can be expressed as a polynomial on the functions $\{\poc(\tau, \sigma ) |\tau \in \mathcal{L}_k \}$ that does not depend on $\sigma$.
It follows that $clP_k$ sits inside an algebraic variety of dimension $|\mathcal{L}_k|$.
We expect that this bound is sharp since this is the case for small values of $k$.

\begin{conjecture}
The feasible region $clP_k$ is full-dimensional inside a manifold of dimension $|\mathcal{L}_k|$.
\end{conjecture}

\subsubsection{Overlap graphs}
\label{sect:ov_gr}
Overlap graphs were already studied in previous works. We give here a brief summary of the relevant literature. The overlap graph $\ValGraph[k]$ is the line graph of the \emph{de Bruijn graph for permutations} of size $k-1$. The latter was introduced in \cite{MR1197444} and further studied in \cite{MR3944621}, where the authors studied universal cycles (sometime also called \emph{de Bruijn cycles}) of several combinatorial structures, including permutations. 
In this case, a universal cycle of order $n$ is a cyclic word of size $n!$ on an ordered alphabet of $N$ letters that contains all the patterns of size $n$ as consecutive patterns. In \cite{MR1197444} it was conjectured (and then proved in \cite{MR2548540}) that such universal cycles always exist when the alphabet is of size $N=n+1$. 
In \cite{asplund2018enumerating} the authors enumerate some families of simple cycles of these graphs.
This starts the undertaking of describing all simple cycles of a explicit de Bruijn graph.
As a consequence, this also gives a description of the vertices of the corresponding cycle polytope.

The existence of Eulerian and Hamiltonian cycles in classical de Bruijn graphs is shown in \cite{MR1197444}.
We remark that applying the same ideas, we can prove the existence of both Eulerian and Hamiltonian cycles in $\ValGraph[k]$.
We further remark that with an Eulerian path in $\ValGraph[k]$, we can construct a permutation $\sigma $ of size $k!+k -1$ such that $\coc (\pi,  \sigma ) = 1 $ for any $\pi\in \SS_k$.
This is a so-called ``superpermutation'' of minimal size for consecutive patterns.

\subsubsection{Polytopes and cycle polytopes}
\label{sect:cyc_pol}

Polytopes associated to graphs have been objects of research for a long time in computer science and graph theory.
A striking example is the \textit{flow polytope}, introduced in \cite{baldoni2008kostant}.
This is a polytope that is associated to a root system and a \textit{flow vector}.
If we specifically take a root system of type $A_n$, this can be described by an undirected graph on $[n]\coloneqq\{1,\dots,n\}$: in this way, if we are given a graph $G=([n], E)$ and a flow vector $\vec{a}\in \R^{n}$, its corresponding flow polytope is in fact defined as
$$\mathcal{F}_G(\vec{a} ) \coloneqq \left\{\vec{x}\in \R^{E} \Bigg| \sum_{ \{j <i\} \in E}\vec{x}_{\{j <i\}} - \sum_{ \{i <j\}  \in E} \vec{x}_{\{i <j\}} = \vec{a}_i, \, i \in [n] \right\}\, . $$

Classical examples of polytopes that are flow polytopes are the \textit{Stanley--Pitman polytope}, also called the \textit{parking function polytope}, and the 
\textit{Chan--Robbins--Yuen polytope}, a polytope on the space of doubly stochastic square matrices.
In \cite{baldoni2008kostant}, the authors obtain formulas for the volume and the number of integer points in its interior.
In particular, they recover a formula for the volume of the Chan--Robbins--Yuen polytope, due to Zeilberger, in his very short paper \cite{zeilberger1999proof}.

In \cite{balas2000cycle} and in \cite{MR986897}, \textit{unrescaled cycle polytopes} (\textit{U-cycle polytopes} for short)\footnote{The U-cycle polytopes and the cycle polytopes introduced in \cref{def:cyclepoly} are intrinsically related to one another, as the vertices of the U-cycle polytope of a directed multigraph $G$ are defined as the incidence vectors of simple cycles of $G$. In \cref{def:cyclepoly}, we \emph{additionally rescale} each of the vertices so that the coordinates sum up to one. The U-cycle polytopes were considered in the literature simply under the name of \textit{cycle polytopes}. However, here, we adopt the name of \emph{cycle polytopes} to our family of polytopes for the sake of simplifying the terminology in this extended abstract.} were introduced.
The U-cycle polytopes for undirected graphs were considered to tackle the Simple Cycle Problem \cite{MR1944484}, that also goes by the name of Weighted Girth Problem \cite{bauer1997circuit}. 
Balas \& Oosten \cite{balas2000cycle} and Balas \& Stephan \cite{balas2009cycle} compute the dimension of the U-cycle polytope of the \emph{complete graph} (that is, the complete directed graph without loops) and describe the faces of co-dimension one of the corresponding polytope.
In comparison, we study and give a dimension theorem for cycle polytopes (instead of U-cycle polytopes) for general directed multigraphs (not restricting to the case of complete graphs). Our results extend some results of Gleiss, Leydold and Stadler \cite{MR2070155}.

\section{First main result: Description of \texorpdfstring{$P_k$}{the feasible region}}
Our first main result is the following.

\begin{theorem}
	\label{thm:main_res}
	$P_k$ is the cycle polytope of the overlap graph $\ValGraph[k]$. Its dimension is $k! - (k-1)!$ and its vertices are given by the simple cycles of $\ValGraph[k]$.
\end{theorem}

In addition, we can also determine the equations that describe the polytope $P_k$ (for that see \cite[Theorem 3.12]{borga2019feasible}).

\medskip

To establish that $P_k=P(\ValGraph[k])$, the first step is the following result.

\begin{proposition}
	\label{prop:P_n_is_covex}
	The feasible region $P_k$ is convex.
\end{proposition}

\begin{proof}
	Since $P_k$ is closed (this is an easy consequence of the fact that $P_k$ is a set of limit points) it is enough to consider rational convex combinations of points in $P_k$, i.e.\ it is enough to establish that for all $\vec{v}_1,\vec{v}_2\in P_k$ and all $ s,t\in\NN$, we have that
	\begin{equation*}
	\frac{s}{s+t}\vec{v}_1+\frac{t}{s+t}\vec{v}_2\in P_k.
	\end{equation*}
	For $\sigma\in\SS$, we define the vector $\pcoc_k( \sigma) \coloneqq (\pcoc(\pi,\sigma))_{\pi\in\SS_k}$. Fix $\vec{v}_1,\vec{v}_2\in P_k$ and $s,t\in\NN$. Since $\vec{v}_1,\vec{v}_2\in P_k$, there exist two sequences $(\sigma^m_1)_{m\in\NN}$, $(\sigma^m_2)_{m\in\NN}$ such that $|\sigma^m_i|\stackrel{m\to\infty}{\longrightarrow}\infty$ and $\pcoc_k( \sigma^m_i)\stackrel{m\to\infty}{\longrightarrow}\vec{v}_i$, for $i=1,2$.
	
	Define $t_m\coloneqq t\cdot |\sigma^m_1|$ and $s_m\coloneqq s\cdot |\sigma^m_2|$ and set $\tau^m$ to be equal to the direct sum\footnote{For $\tau\in\SS_m$ and $\sigma\in\SS_n$, the direct sum of $\tau$ and $\sigma$ is the permutation  $\tau(1)\dots\tau(m)(\sigma(1) + m)\dots(\sigma(n)+m)$.} of $s_m$ copies of $\sigma_1^m$ and $t_m$ copies of $\sigma_2^m$.
	
	We note that for every $\pi\in\SS_k$, we have
	\begin{equation*}
	\coc(\pi,\tau^m)=s_m\cdot\coc(\pi,\sigma^m_1)+t_m\cdot\coc(\pi,\sigma^m_2)+Er,
	\end{equation*}
	where $Er\leq(s_m+t_m-1)\cdot |\pi|$. This error term comes from the number of intervals of size $|\pi|$ that intersect the boundary of some copies of $\sigma^m_1$ or $\sigma^m_2$. Hence
	\begin{equation*}
	\begin{split}
	\pcoc(\pi,\tau^m)&=\frac{s_m\cdot|\sigma^m_1|\cdot\pcoc(\pi,\sigma^m_1)+t_m\cdot|\sigma^m_2|\cdot\pcoc(\pi,\sigma^m_2)+Er}{s_m\cdot|\sigma^m_1|+t_m\cdot |\sigma^m_2|}\\
	&=\frac{s}{s+t}\pcoc(\pi,\sigma^m_1)+\frac{t}{s+t}\pcoc(\pi,\sigma^m_2)+O\left(|\pi|\left(\tfrac{1}{|\sigma^m_1|}+\tfrac{1}{|\sigma^m_2|}\right)\right).
	\end{split}
	\end{equation*}
	As $m$ tends to infinity, we have $$\pcoc_k(\tau^m)\to\frac{s}{s+t}\vec{v}_1+\frac{t}{s+t}\vec{v}_2,$$ 
	since $|\sigma^m_i|\stackrel{m\to\infty}{\longrightarrow}\infty$ and $\pcoc_k( \sigma^m_i)\stackrel{m\to\infty}{\longrightarrow}\vec{v}_i$, for $i=1,2$. Noting also that
	$$|\tau^m|\to\infty,$$
	we can conclude that $\tfrac{s}{s+t}\vec{v}_1+\tfrac{t}{s+t}\vec{v}_2\in P_k$. This ends the proof.
\end{proof}

Using \cref{prop:P_n_is_covex} we can show that $P_k=P(\ValGraph[k])$ using a surjective correspondence between permutations of size $m+k-1$ and walks of length $m$ in the overlap graph $\ValGraph[k]$.
This correspondence is built in such a way that the vertices on the walk are in bijection with the consecutive patterns of size $k$ appearing in the corresponding permutation (for that we refer to \cite[Sections 3.3 and 3.4]{borga2019feasible}).

We use this correspondence as follows.
First we show that $P(\ValGraph[k])\subseteq P_k$: Thanks to \cref{prop:P_n_is_covex} it is enough to show that for each simple cycle $\mathcal{C}\in\ValGraph[k]$, the vector $\vec{e}_{\mathcal{C}}$ belongs to $P_k$. For that we consider the sequence of walks $(w_m)_{m\in\NN}$ obtained concatenating $m$ copies of $\mathcal{C}$, and we show that a corresponding sequence of permutations $(\sigma^m)_{m\geq 1}$ satisfies $\pcoc_k(\sigma^m) \to \vec{e}_{\mathcal{C}}$.
Second we show that $ P_k\subseteq P(\ValGraph[k])$: For any vector $\vec{v}\in P_k$, we consider a sequence of permutations $(\sigma^m)_{m\in\NN}$ such that $\pcoc_k(\sigma^m)\to\vec{v}$. This sequence of permutations corresponds to a sequence of walks in $\ValGraph[k]$.
By decomposing each of these walks into simple cycles (plus a small tail) we show that the original vector $\vec{v}$ is arbitrarily close to $P(\ValGraph[k])$.

Finally, the result on the dimension, the description of the vertices and the equations describing $P_k$ are a consequence of general results for cycle polytopes (see details for the computation of the dimension in \cref{sect:sec_res}, and for the remaining features in \cite{borga2019feasible}).

\section{Second main result: Dimension and faces of cycle polytopes}
\label{sect:sec_res}

In \cite{borga2019feasible}, we establish a full description of the cycle polytope of a general directed multigraph (see \cref{thm:dimfullgr} below) and of its face poset (see \cref{cor:facestruct} below).
We give here these descriptions omitting the proofs.

We say that a directed multigraph is \textit{full} if every edge is contained in a cycle.

\begin{theorem}\label{thm:dimfullgr}
If $G$ is a directed multigraph and $H\subseteq G$ is its largest full subgraph, then the dimension of the polytope $P(G)$ is
$$\dim P(G) = |E(H)| - |V(H)| + |\{ \text{ connected components of } H \} |- 1\, . $$

In particular, if $G$ is a strongly connected graph, then $\dim P(G) = |E(G)| - |V(G)|$.
\end{theorem}

Additionally,  we are able to describe the equations that define $P(G)$ (for a precise statement see \cite[Proposition 2.6]{borga2019feasible}).
We highlight that \cref{thm:dimfullgr}, together with $P_k = P(\ValGraph[k]) $, yields the dimension of the feasible region $P_k$ stated in \cref{thm:main_res}.

The proof of \cref{thm:dimfullgr} follows from a generalization of some previous results of Gleiss, Leydold and Stadler \cite{MR2070155}. We refer the reader to \cite[Section 2.2]{borga2019feasible} for further details.

We also give an order-preserving bijection between faces of $P(G)$ and full subgraphs of $G$ (ordered by inclusion of edge sets), giving a description of the face poset of $P(G)$.

\begin{theorem}
	\label{cor:facestruct}
	The face poset of $P(G)$ is isomorphic to the poset of full subgraphs of $G$ according to the following identification:
	$$H \mapsto P(G)_H \coloneqq \{\vec{x}\in P(G) | x_e = 0 \text{ for } e\not\in E(H) \} \, .$$
	Further, if we identify $P(H)$ with its image under the canonical injection $\R^{E(H)} \hookrightarrow \R^{E(G)}$, we have that $P(H) =  P(G)_H$.
	
	In particular, $\dim P(G)_H =|E(H)| - |V(H)| + |\{ \text{ connected components of } H \} |- 1$.
\end{theorem}

The above result is proved in \cite[Theorem 2.7]{borga2019feasible} and illustrated in \cref{fig:facestruct}. 

\begin{figure}
	\begin{minipage}[c]{0.49\textwidth}
		\centering
		\includegraphics[scale=0.45]{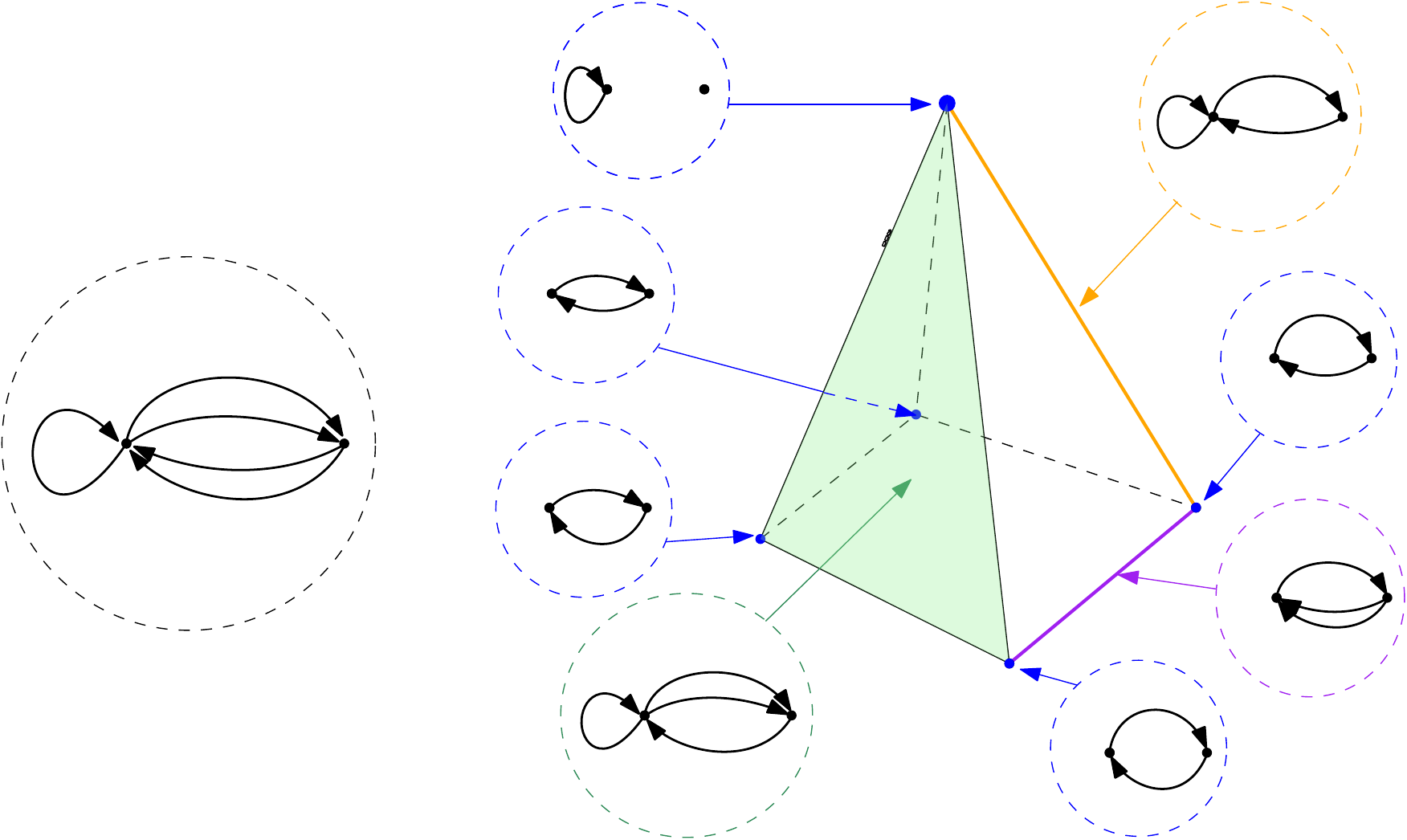}
	\end{minipage}
	\begin{minipage}[c]{0.49\textwidth}
		\caption{\textbf{Left:} Inside the dashed black ball, a graph $G$ with two vertices and five oriented edges, including a loop. \textbf{Right:} The associated cycle polytope $P(G)$, that is a pyramid with squared base. The blue dashed balls indicate the five vertices of the polytope, and correspond to the simple cycles. We also underline two edges of the polytope (in purple and orange respectively) and a face (in green), as well as their corresponding full subgraphs. \label{fig:facestruct}}
	\end{minipage}
\end{figure}

\section{Third main result: Mixing classical and consecutive patterns}
\label{sect:mixing}
We saw in \cref{sect:feas_reg} that the feasible region $clP_k$ for classical pattern occurrences has been studied in several papers. 
In this extended abstract, the feasible region $P_k$ of limiting points for consecutive pattern occurrences was described. A natural question is the following: what is the feasible region if we mix classical and consecutive patterns?
We reduce this question to the description of the original feasible regions separately in the following result.

\begin{theorem} [{\cite[Theorem 4.1]{borga2019feasible}}] \label{thm:mixing}
	For any two points $\vec{v}_1 \in P_k$, $ \vec{v}_2\in clP_k$, consider two sequences
	$(\sigma^m_1)_{m\in\NN}\in{\SS}^{\NN}$ and $(\sigma^m_2)_{m\in\NN}\in{\SS}^{\NN}$ such that 
	\begin{equation*}
	|\sigma^m_1| \to \infty,\quad \left(\pcoc(\pi, \sigma^m_1)\right)_{\pi\in\SS_k} \to \vec{v}_1,\quad\text{ and }\quad
	|\sigma^m_2| \to \infty,\quad \left(\poc(\pi, \sigma^m_2 )\right)_{\pi\in\SS_k} \to \vec{v}_2.
	\end{equation*}
	Then the sequence $(\sigma^m_3)_{m\in\NN}$ defined\footnote{An \emph{interval} in a permutation $\sigma=\sigma(1)\dots\sigma(n)$ is a factor $\sigma(i)\sigma(i+1)\dots\sigma(j)$, for $j\geq i$, such that the set $\{\sigma(i),\sigma(i+1),\dots,\sigma(j)\}$ forms an interval of $\NN$. If $\sigma$ is a permutation of size $n$ and $\tau_1,\dots,\tau_n$ are permutations, then the \emph{inflation} of $\sigma$ by $\tau_1,\dots,\tau_n$ (denoted by $\sigma[\tau_1,\dots,\tau_n]$) is the permutation with disjoint intervals $\tau'_1,\dots,\tau'_n$, each respectively isomorphic to $\tau_1,\dots,\tau_n$, from left to right, whose relative order is given by $\sigma$.} by $\sigma^m_3\coloneqq\sigma^m_2[\sigma^m_1,\dots,\sigma^m_1]$, for all $m\in\NN,$
	satisfies
	\begin{equation}\label{eq:goaloftheproof}
	|\sigma^m_3|\to \infty, \quad\left(\pcoc(\pi, \sigma^m_3)\right)_{\pi\in\SS_{k}} \to \vec{v}_1 \quad\text{and}\quad \left(\poc(\pi, \sigma^m_3 )\right)_{\pi\in\SS_k} \to \vec{v}_2.
	\end{equation}	
\end{theorem}

This result shows a sort of independence between classical and consecutive patterns, in the sense that knowing the limiting proportion of classical patterns of a certain sequence of permutations imposes no constraints for the limiting proportion of consecutive patterns and \textit{vice versa}.

\begin{observation}
	In \cref{thm:carac_occ,thm:carac_cocc} we saw that the proportion of occurrences (resp.\ consecutive occurrences) in a sequence of permutations $(\sigma^m)_{m\in\NN}$ characterizes the permuton limit (resp.\ Benjamini--Schramm limit) of the sequence. From \cref{thm:mixing}, we can construct a sequence of permutations where the permuton limit is the decreasing diagonal and the Benjamini--Schramm limit is the classical increasing total order on the integer numbers.
	
	We remark that a particular instance of this ``independence phenomenon'' for local/scaling limits of permutations was recently observed by Bevan, who pointed out in \cite{bevan2019permutations} that ``the knowledge of the local structure of uniformly random permutations with a specific fixed proportion of inversions reveals nothing about their global form''. Here, we prove that this is a \emph{universal phenomenon} which is not specific to the framework studied by Bevan.
\end{observation}

\begin{proof}[Proof of \cref{thm:mixing}]
The size of $\sigma^m_3$ trivially tends to infinity.
We skip the proof of the limit $\left(\pcoc(\pi, \sigma^m_3)\right)_{\pi\in\SS_{k}} \to \vec{v}_1$ since it involves similar techniques as the ones used in the proof of \cref{prop:P_n_is_covex}.
Finally, for the limit $\left(\poc(\pi, \sigma^m_3 )\right)_{\pi\in\SS_k} \to \vec{v}_2,$ note that setting $n=|\sigma^m_3|$ and $k=|\pi|$,
\begin{equation}
\label{eq:step_1}
	\poc(\pi,\sigma^m_3) = \frac{\occ(\pi,\sigma^m_3)}{\binom{n}{k}} \, 
	=\mathbb{P} \left(\pat_{\bm I}(\sigma^m_3)=\pi \right),
\end{equation}
where $\bm I$ is a random set, uniformly chosen among the $\binom{n}{k}$ subsets of $[n]$ with $k$ elements (we denote random quantities in \textbf{bold}). Let $E^m$ be the event that the random set $\bm I$ contains two indices $\bm i,\bm j$ of $[|\sigma^m_3|]$ that belong to the same copy of $\sigma^m_1$ in $\sigma^m_3$. We have that $\mathbb{P} \left(\pat_{\bm I}(\sigma^m_3)=\pi \right)$ rewrites as
\begin{equation}
\label{eq:step_2}
	\mathbb{P} \left(\pat_{\bm I}(\sigma^m_3)=\pi |E^m\right)\cdot\mathbb{P}\left(E^m\right)+\mathbb{P} \left(\pat_{\bm I}(\sigma^m_3)=\pi |(E^m)^C\right)\cdot\mathbb{P}\left((E^m)^C\right),
\end{equation}
where $(E^m)^C$ denotes the complement of the event $E^m$. We claim that
\begin{equation}
\label{eq:step_3}
	\mathbb{P}\left(E^m\right)\leq k(k-1)/(2\cdot|\sigma^m_2|)\to0.
\end{equation}
Indeed, the factor $k(k-1)/2$ counts the number of pairs $i,j$ in a set of cardinality $k$ and the factor $\frac{1}{|\sigma^m_2|}$ is an upper bound for the probability that given a uniform two-element set $\left\{\bm i,\bm j\right\}$ then $\bm i, \bm j$ belong to the same copy of $\sigma^m_1$ in $\sigma^m_3$ (recall that there are $|\sigma^m_2|$ copies of $\sigma^m_1$ in $\sigma^m_3$). Note also that 
\begin{equation}
\label{eq:step_4}
\mathbb{P} \left(\pat_{\bm I}(\sigma^m_3)=\pi |(E^m)^C\right)=\poc(\pi,\sigma^m_2)\to \vec{v}_2.
\end{equation}
 Using \cref{eq:step_1,eq:step_2,eq:step_3,eq:step_4}, we obtain that $\left(\poc(\pi, \sigma^m_3 )\right)_{\pi\in\SS_k} \to \vec{v}_2.$ 
\end{proof}
\bibliographystyle{abbrv}
\bibliography{bibli}
\end{document}